\documentclass{amsart}
\usepackage{amssymb}
\usepackage{amsthm}
\usepackage{eucal}

\newcommand{\R}{\mathbb R}

\newcommand{\sgn}{\text{sgn}}

\newtheorem*{remarks}{Remarks}

\newtheorem{theorem}{Theorem}[section]
\newtheorem{proposition}[theorem]{Proposition}

\newtheorem{lemma}[theorem]{Lemma}

\newtheorem{definition}[theorem]{Definition}

\title[The  ${\text r}$BO equation]{Well-posedness and ill-posedness results  for the regularized  Benjamin-Ono equation in weighted Sobolev spaces}
\author{Germ\'an Fonseca}
\address[G. Fonseca]{Departamento  de Matem\'aticas\\
Universidad Nacional de Colombia\\ Bogota\\Colombia}
\email{gefonsecab@unal.edu.co}

\author{Guillermo Rodríguez-Blanco}
\address[G. Rodriguez]{Departamento  de Matem\'aticas\\
Universidad Nacional de Colombia\\ Bogota\\Colombia}
\email{grodriguezb@unal.edu.co}
\author{Wilson Sandoval}
\address[W. Sandoval]{Departamento  de Matem\'aticas\\
Universidad Nacional de Colombia\\ Bogota\\Colombia}
\email{wsandoval@unal.edu.co}

\date{}
\begin{document}
\begin{abstract} We consider the initial value problem associated to the regularized Benjamin-Ono equation, rBO.
Our aim is to establish  local and global well-posedness results in  weighted Sobolev
spaces via contraction principle. We also prove a unique continuation property that implies that arbitrary polinomial type decay is not preserved yielding  sharp results regarding well-posedness of the initial value problem in most weighted Sobolev spaces. 
\end{abstract}

\maketitle

\section{Introduction.}
In this work we shall study  the initial value problem (IVP)
for the   regularized Benjamin-Ono (rBO) equation
\begin{equation}\label{rBO}
\begin{cases}
\partial_t u+\partial_x u+\mathcal H\partial^2_{xt} u+u\partial_x u = 0, \qquad t, x\in \R,\;\;\;\; \\
u(x,0) = \varphi(x),
\end{cases}
\end{equation}
where  $\mathcal H$ denotes the Hilbert transform,
\begin{equation*}
\mathcal H f(x)=\frac{1}{\pi}\lim_{\epsilon\downarrow 0}\int\limits_{|y|\ge \epsilon} \frac{f(x-y)}{y}\,dy=(-i\,\sgn(\xi) \widehat{f}(\xi))^{\vee}(x).
\end{equation*}
The regularized BenjaminOno equation, rBO,  models propagation of long-crested waves at
the interface between two immiscible fluids.  In particular this equation is appropiate to describe the
pycnocline in the deep ocean, and the two-layer system created by the inflow of fresh water from a
river into the sea ( see \cite{Kal} and references therein). A related equation associated with the same type of phenomena is the Benjamin-Ono equation

\begin{equation}\label{BO}
\partial_t u+\mathcal H\partial^2_{xx} u+u\partial_x u = 0, \qquad t, x\in \R,\;\;\;\;
\end{equation}
which was introduced by Benjamin \cite{Be} and Ono \cite{On}  as a model for long internal gravity waves in deep stratified fluids. Both equations admit solitary wave solutions but it is well known that BO equation defines a completely integrable system, however,  for the rBO equation,  numerical simulations suggest that this property does not hold, see \cite{Kal} and \cite{BoKa}. 

On the other hand it is interesting to notice that for both equations the dispersive part is given by a non-local operator involving the Hilbert transform , however, it is possible to find solutions with Picard iterations in the Sobolev spaces, $H^s(\R) = \left(1-\partial^2_x\right)^{-s/2} L^2(\R),$ for the rBO equation but it was recently shown that this scheme can not be applied for the BO equation, see \cite{MoSaTz} and \cite{KoTz2}. Numerical studies take advantage of this fact and give evidence of the better suitability of the rBO equation over the BO equation for modelling purposes, see \cite{Kal}. 

In \cite{AlBo} and \cite{BoKa}  the initial value problem \ref{rBO} was shown to be globally well posed in $H^{s}(\R), s\geq \frac32.$ Recently, in \cite{An}, rBO equation  was considered in the periodic and in the continuos setting. There it was shown that \ref{rBO} was locally and globally well-posed  in $H^{s}, s> \frac12.$  It is worth to mention that in \cite{An} it is proved that when considering negative indexes, {\it i,e} $s<0,$ the map data-solution for the rBO flow is not $C^2$ and therefore Picard's iteration fails for those rough Sobolev spaces.

Our goal in this paper is to obtain sharp well-posedness results in weighted Sobolev spaces of $L^2-$ type, $\Im_{s,r} = H^s(\mathbb{R})\cap L^2_r(\mathbb{R})$, where $L_r^2(\mathbb{R})=\{f\in L^2(\mathbb{R})\vert$ $  |x|^rf \in L^2(\mathbb{R}) \}$, $r\in\R $. It follows that $\Im_{s,r} $ is a Banach space endowed with the norm  $\Vert f\Vert_{s,r}^2=\Vert f\Vert_s^2 +\Vert f\Vert_{L_{r}^2}^2. $

Regarding the BO equation \ref{BO}, it was recently shown in \cite{FoPo} that arbitrary decay is not preserved, and that indeed there is an upper limit bound for the decay of solutions in order to guarantee well-posedness in those spaces which expressed in terms of the index $r$ tells us that $r$ should be less than $7/2$ which contrasts with another wave propagation famous model, the Korteweg-de Vries equation  \cite{KdV}, for which the Schwartz class is preserved, \cite{Ka}. These type of results are of special interest from the point of view of the qualitative properties of the solitary waves for the respective flow. It is well known that for the BO the profile of its solitary waves has mild decay, $r<\frac32,$ whereas for the KdV the solitary waves decay exponentially. 

The discussion for the rBO equation in these weighted Sobolev spaces is contained in our main theorems

\begin{theorem}\label{Main1} Let $\varphi \in \Im_{s,r}(\mathbb{R}), s>1/2, r=1,2,$ then there  exists a unique  $u\in C([0,\infty);\Im_{s,r})$ solution of the initial value problem (\ref{rBO}) such that, $\partial _tu\in C([0,\infty);L^2(\mathbb{R})).$
\label{integercase}
\end{theorem}

\begin{theorem}\label{Main2} Let $\varphi \in \Im_{s,r}(\mathbb{R})$ with $ s>1/2$ and $0<r<5/2$ then there  exists a unique  $u\in C([0,\infty);\Im_{s,r})$ solution of the initial value problem (\ref{rBO}) such that, $\partial _tu\in C([0,\infty);L^2(\mathbb{R})).$
\end{theorem}

\begin{theorem}\label{Main3} Let $u\in C([0,T]:\Im_{s,2})$ be the  solution of the initial value problem \ref{rBO} with  initial data, $\varphi,$ having mean value $\int \varphi(x)\,dx\geq 0.$ Then if at times $t_1=0<t_2<T$ it holds that $u(t_j)\in \Im_{s,\frac52}, j=1,2$ then $u$ must be identically zero
\end{theorem}
\begin {remarks} \hskip10pt
\begin{enumerate}
\item[a)] Theorem \ref{Main1} refers  to integer value powers  of the weights whilst Theorem \ref{Main2} deals with general real value powers of them. We choose to present this way our results since in the latter case the proof is a bit  more involved and relies in some continuity properties of the Hilbert transform. We remark that for the BO equation similar results to Theorem \ref{Main1} and Theorem \ref{Main3} were first established by Iorio in \cite{Io1} and \cite{Io2}, and those in Theorem \ref{Main2} and Theorem \ref{Main3} were recently obtained by Fonseca, Ponce and Linares in \cite{FoPo} and  \cite{GFFLGP1} and more generally for the dispersion generalized Benjamin-Ono equation by the latter authors in \cite{GFFLGP2}. 
\item[b)] Theorem \ref{Main3} establishes that the initial value problem is ill-posed whenever the decay rate $r$ of the initial data is or goes beyond $\frac52.$ In particular, the result in Theorem  \ref{Main2} is sharp. Also, we notice that this unique continuation result requires information at only two different times in a similar way to what happens in the case of the generalized KdV and non-linear Scr\"odinger equations, see \cite{EKPV1} and \cite{EKPV2}. For the BO and dispersion generalized BO equations, information at three different times is indeed required, see \cite{GFFLGP1} and \cite{GFFLGP2}.
\item[c)] The proof of Theorem \ref{Main3} implies that if $\widehat{\varphi}(0)=\int_\infty^\infty \,\, \varphi (x)\,dx<0,$ then the $L^2$-norm of the solution is conserved and $\|u(t)\|^2_0=\|\varphi\|_0^2=-2\widehat{\varphi}(0)$, and therefore under more general hypothesis on the initial data $\varphi,$ ill-posedness for the I.V.P \ref{rBO} holds.
     
\end{enumerate}
\end{remarks}

\section{Preliminary results}

\subsection{Wellposedness in $H^s(\R)$, $s> 1/2$}


The equation \eqref{rBO} is equivalent to the integral equation
\begin{equation}\label{IE}
 u(t)=e^{tA}\varphi+\int_0^t e^{(t-\tau)A}f(u(\tau))d\tau,
\end{equation}
where, $A =-\partial_x(1+\mathcal{H}\partial_x)^{-1}$ and $f(u)=-\dfrac{1}{2} \partial_x (1+\mathcal{H}\partial_x)^{-1} (u^2).$

\begin{theorem}\label{theorem1}  
 The Cauchy problem \ref{rBO} is l.w.p. in $H^s(\R)$, $s> 1/2.$
\end{theorem}
\begin{proof} This result is a consequence from \ref{IE}, Banach's fixed point Theorem and that $f$ is locally Lipchitz, for $s>1/2$, because $H^s$ is a Banach Algebra for $s>1/2$.
\end{proof}

\begin{lemma}\label{CC}  
Suposse that $\varphi \in H^s$ for $s>\frac{1}{2}$ and let $u \in C([-T,T] , H^s(\R))$ be the solution of (\ref{rBO}),  then $$\Vert u(t) \Vert_{\frac{1}{2}}\sim\Vert \varphi \Vert_{\frac{1}{2}}.$$
\end{lemma}
\begin{proof} 
The equation (\ref{rBO}) implies that

\begin{equation}
(1+\mathcal{H}\partial_x)u_t= - \partial_xu-\frac{1}{2}\partial_xu^2.
\end{equation}

Now 
$$\left\Vert u\right\Vert _{\frac{1}{2}}^{2}\sim\left\langle \tilde{J}^{\frac{1}{2}}u\ ,\ \tilde{J}^{\frac{1}{2}}u\right\rangle,$$ 
 
where, $\tilde{J}^{s}=(1+\mathcal H\partial_{x})^{s}$.\\

Therefore, it easily follows  that

\begin{align} \frac{d}{dt}\left\langle \tilde{J}^{\frac{1}{2}}u\ ,\ \tilde{J}^{\frac{1}{2}}u\right\rangle \nonumber
 & =2\left\langle (1+\mathcal{H}\partial_{x})^{\frac{1}{2}}u\ ,\ (1+\mathcal{H}\partial_{x})^{\frac{1}{2}}\big(-\partial_{x}(1+\mathcal{H}\partial_{x})^{-1}u-\frac{1}{2}\partial_{x}(1+\mathcal{H}\partial_{x})^{-1}u^{2}\big)\right\rangle \nonumber\\
 & =2\left\langle (1+\mathcal{H}\partial_{x})u\ ,\ -\frac{\partial_{x}u}{(1+\mathcal{H}\partial_{x})}-\frac{\partial_{x}u^{2}}{2(1+\mathcal{H}\partial_{x})}\right\rangle \nonumber \\
 & =2\left\langle u\ ,\ -\partial_{x}u-\frac{1}{2}\partial_{x}u^{2}\right\rangle \nonumber\\
 & =-2\left\langle u\ ,\ \partial_{x}u\right\rangle -\left\langle u\ ,\ \partial_{x}u^{2}\right\rangle \label{conservada}\\
 & =0. \nonumber
\end{align}
This implies the result.
\end{proof}
 Next theorem in \cite{An} shows that the I.V.P \ref{rBO} is globally well-posed in $H^s(\R), s>1/2$.  Since the proof in \cite{An} was only carried out in detail for the periodic case, we find interesting to include its proof in the continous setting. The key point is to obtain {\it a priori} estimates of the Sobolev norm in $H^s(\R), s>1/2$ with the help of the {\it a priori} bound of the $H^{\frac{1}{2}}$ norm in Lemma \ref{CC} and the help of the Kato-Ponce commutator, \cite{KaPo},

\begin{equation*}
|fg|_{s,p}\leq c(\|f\|_{L^\infty} |g|_{s,p}+\|g\|_{L^\infty} |f|_{s,p})\,\,\qquad{\text for} \qquad s>0, \,\,\,1<p<\infty,
\end{equation*}
and the Brezis-Gallouet's inequality, \cite{BrGa},  which in dimension one reads
 
\begin{equation*}
\|f\|_{L^\infty}\leq c(1+\sqrt{log(1+\|f\|_s)}\|f\|_{\frac12}), \,\,{\text for}\,\,s>1/2.
\end{equation*}
 
This last inequality was also used in the same spirit by Ponce to show global-well-posedness of the Benjamin-Ono equation in $H^{s}(\R), s\geq \frac{3}{2},$ see \cite{Po}.

\begin{theorem}\label{theorem2}  
 The Cauchy problem \ref{rBO} is g.w.p. in $H^s(\R)$, $s> 1/2.$
\end{theorem}
\begin{proof}

Let $\varphi \in H^{s}(\mathbb{R})$, the integral equation
\begin{equation}\label{solu}
 u(t)=e^{tA}\varphi+\int_0^t e^{(t-\tau)A}f(u(\tau))d\tau,
\end{equation}
implies that

\begin{align}
\Vert u(t) \Vert_{s} \leq & \Vert \varphi\Vert_{s} +  \dfrac{1}{2}\int_0^t \Vert u(\tau) \Vert_{\infty} \Vert u(\tau) \Vert_{s} d\tau \nonumber\\
&\leq \Vert \varphi\Vert_{s} + C_0\int_0^t \Big (1+\sqrt{\log(1+\Vert u(\tau) \Vert_{s})}\Big) \Vert u(\tau) \Vert_{s}  d\tau =:\Psi(t),\label{global}
\end{align}

where $C_0$ depends only $\Vert \varphi \Vert_{\frac{1}{2}}$.\\

The lemma above and this inequality imply that
\begin{align*}
\Psi '(t) &= C_0\Big (1+\sqrt{\log(1+\Vert u(t) \Vert_{s})}\Big)\Vert u(t) \Vert_{s}\\
&\leq\ C_0\Big (1+\sqrt{\log(1+ \Psi(t))}\Big)\Psi (t)\\
&\leq C_0\Big (1+\log(1+ \Psi(t))\Big)\Psi (t).
\end{align*}
Then, there exists  $C_1>0$ such that, 

\begin{equation*}
\frac{d}{dt}\log(1+\log(1+\Psi(t)))\leq C_1,
\end{equation*}
and  hence, there are constants  $C_2 >0$ and $C_3 >0$ such that for every $t \in [-T , T]$, $\Vert u(t)\Vert_{s} \leq e^{C_{2}\large{e}^{C_3t}}$. 

\end{proof}

\subsection{Group estimates in weighted spaces}

\vskip .1in

In order to obtain the group estimates, next lemma provides some formulae for derivatives of the unitary group associated to the rBO equation in Fourier space. They  easily follow from a direct computation.  

\begin{lemma}\label{derivatives} 
Let $F(t,\xi)= \LARGE{e}^{b(\xi)t}$, where $b(\xi)=\dfrac{-i\xi}{1+|\xi|}.$ Then,
\begin{align}\partial_{\xi}F(t,\xi)  =&(-it)(1+|\xi|)^{-2}F(t,\xi)\\
\partial_{\xi}^{2}F(t,\xi) =&2it \sgn(\xi)(1+|\xi|)^{-3}F(t,\xi)+(-it)^{2}(1+|\xi|)^{-4}F(t,\xi)\label{der1}\\
\partial_{\xi}^{3}F(t,\xi)  =&4it\delta-6it(1+|\xi|)^{-4}F(t,\xi)-2(-it)^{2}[\sgn(\xi)+2](1+|\xi|)^{-5}F(t,\xi) \label{tercera}\nonumber\\
&+ (-it)^{3}(1+|\xi|)^{-6}F(t,\xi)\\
\partial_{\xi}^{4}F(t,\xi)  =& 4it\delta'+4t^{2}\delta+24it \sgn(\xi)(1+|\xi|)^{-5}F(t,\xi) \nonumber \\
&+4(-it)^{2}[4+5\sgn(\xi)](1+|\xi|)^{-6}F(t,\xi) \\
&-2(-it)^{3}[\sgn(\xi)+5](1+|\xi|)^{-7}F(t,\xi)+(-it)^{4}(1+|\xi|)^{-8}F(t,\xi)\nonumber \\
\partial_{\xi}^{5}F(t,\xi)  =& 4it\delta''+4t^{2}\delta'+(48it-40t^{2}-4it^{3})\delta-120it(1+|\xi|)^{-6}F(t,\xi) \nonumber \\ 
&+(it)^{2}[-120\sgn(\xi)-120](1+|\xi|)^{-7}F(t,\xi) \nonumber\\
&+(it)^{3}[-90\sgn(\xi)-30](1+|\xi|)^{-8}F(t,\xi) \label{der4} \\
 & +(it)^{4}[-10\sgn(\xi)-10](1+|\xi|)^{-9}F(t,\xi)+(-it)^{5}(1+|\xi|)^{-10}F(t,\xi).\nonumber 
\end{align}
Moreover, for $j\geq 5$ the j-derivative for $F(t,\xi), F_j(t,\xi)$, has the form:\\
 
\begin{align}
 F_j(t,\xi)=\partial_{\xi}^jF(t,\xi)=& 4it\delta^{(j-3)} + 4t^2\delta^{(j-4)} + \sum_{k=0}^{j-5}p_k(t)\delta^{k} 
 + (-it)j!(1+|\xi|)^{-j-1}F(t,\xi) \nonumber \\ 
 &  +  \sum_{k=2}^{j-1}(it)^{k}(1+|\xi)^{-j-k}[a_k \sgn(\xi)-b_k] + (-it)^{j}(1+ |\xi|)^{-2j}F(t, \xi),\label{dergen}
 \end{align}
where $\delta$ is the Dirac distribution, $p_k(t)$,  is a polynomial, $a_k$ y $b_k$ are constants that depend on $k.$
\end{lemma}

\begin{theorem} \label{pesosgrupo}
If  $E(t)=e^{tA}$, for $s\in \mathbb{R}$, $r\in \mathbb{N}$ $s\geq r$, we have:

\begin{enumerate}
\item[1.] If $r=0,1, 2$, 
\begin{equation}\label{pc}
 \Vert E(t)\varphi\Vert_{\Im_{s,r}} \leq P_r(t) \Vert \varphi \Vert _{\Im_{s,r}}
 \end{equation} 
 where $P_r(t)$ is a polynomial of grade $r$ with positive coeficients.
\item [2.]If $r\geq 3 $ and $\varphi \in \Im_{s,r}$, \ \ $E\in C([0,\infty) ; \Im_{s,r} )$, if and only if, 
\begin{equation}
(\partial^j_{\xi}\widehat{\varphi})(0)=0, \ \ \ \ \  j= 0,1,2.... r-3.
\end{equation}
In this case, a similar estimate to (\ref{pc}) holds.
\end{enumerate}
\end{theorem}

\begin{proof} We use Leibniz rule and  lemma \ref{derivatives} to conclude that
\begin{align*}
\Vert E(t)\varphi\Vert_{s,0}^{2}= & \Vert E(t)\varphi\Vert_{s}^{2}+\Vert E(t)\varphi\Vert_{L_{0}^{2}}^{2}\\
= & \Vert\varphi\Vert_{s}^{2}+\Vert\varphi\Vert_{0}^{2}.
\end{align*}

\newpage

For $r=1$, we obtain

\begin{align*}
\Vert E(t)\varphi\Vert_{s,1}^{2}= & \Vert E(t)\varphi\Vert_{s}^{2}+\Vert E(t)\varphi\Vert_{L_{1}^{2}}^{2}\\
= & \Vert\varphi\Vert_{s}^{2}+\int_{-\infty}^{\infty}x^{2}\vert E(t)\varphi\vert^{2}dx\\
= & \Vert\varphi\Vert_{s}^{2}+\int_{-\infty}^{\infty}\vert xE(t)\varphi\vert^{2}dx\\
= & \Vert\varphi\Vert_{s}^{2}+\int_{-\infty}^{\infty}\vert\partial_{\xi}(F(t,\xi)\widehat{\varphi})\vert^{2}d\xi\\
= & \Vert\varphi\Vert_{s}^{2}+\int_{-\infty}^{\infty}\Bigl\vert\frac{-it}{(1+\vert\xi\vert)^{2}}F(t,\xi)\widehat{\varphi}+F(t,\xi)\partial_{\xi}\widehat{\varphi}\Bigr\vert^{2}d\xi\\
\leq & \Vert\varphi\Vert_{s}^{2}+t^{2}\int_{-\infty}^{\infty}\vert\widehat{\varphi}\vert^{2}d\xi+\int_{-\infty}^{\infty}\vert\partial_{\xi}\widehat{\varphi}\vert^{2}d\xi\\
\leq & \Vert\varphi\Vert_{s}^{2}+t^{2}\Vert\varphi\Vert_{0}^{2}+\Vert x\varphi\Vert_{0}^{2}\\
\leq & \Vert\varphi\Vert_{s}^{2}+(1+t^{2})\Vert\varphi\Vert_{L_{1}^{2}}^{2}.
\end{align*}

For $r=2$, we have

\begin{align*}
\Vert E(t)\varphi\Vert_{s,2}^{2}= & \Vert E(t)\varphi\Vert_{s}^{2}+\Vert E(t)\varphi\Vert_{L_{2}^{2}}^{2}\\
\leq & \Vert\varphi\Vert_{s}^{2}+\int_{-\infty}^{\infty}\vert\partial_{\xi}^{2}(F(t,\xi)\widehat{\varphi})\vert^{2}d\xi\\
\leq & \Vert\varphi\Vert_{s}^{2}+\int_{-\infty}^{\infty}\Bigl\vert\Big(\frac{2it\sgn(\xi)}{(1+\vert\xi\vert)^{3}}F(t,\xi)-\frac{t^{2}}{(1+\vert\xi\vert)^{4}}F(t,\xi)\Big)\widehat{\varphi}\Bigr\vert^{2}d\xi\\
 & +\int_{-\infty}^{\infty}\Big\vert\frac{it}{(1+\vert\xi\vert)^{2}}F(t,\xi)\partial_{\xi}\widehat{\varphi}\Big\vert^{2}d\xi+\int_{-\infty}^{\infty}\vert F(t,\xi)\partial_{\xi}^{2}\widehat{\varphi}\vert^{2}d\xi\\
\leq & \Vert\varphi\Vert_{s}^{2}+(4t^{2}+t^{4})\int_{-\infty}^{\infty}\vert\widehat{\varphi}\vert^{2}d\xi+t^{2}\int_{-\infty}^{\infty}\vert\partial_{\xi}\widehat{\varphi}\vert^{2}d\xi+t^{2}\int_{-\infty}^{\infty}\vert\partial_{\xi}^{2}\widehat{\varphi}\vert^{2}d\xi\\
\leq & \Vert\varphi\Vert_{s}^{2}+(4t^{2}+t^{4})\Vert\varphi\Vert_{0}^{2}+t^{2}\Vert x\varphi\Vert_{0}^{2}+\Vert x^{2}\varphi\Vert_{0}^{2}.
\end{align*}
 
For $r=3$, we use that $t\delta \widehat{\varphi}=t\delta \widehat{\varphi}(0)=0$, and therefore

\begin{align*}
\Vert E(t)\varphi\Vert_{s,3}^{2}\leq & \Vert\varphi\Vert_{s}^{2}+\int_{-\infty}^{\infty}\vert\partial_{\xi}^{3}(F(t,\xi)\widehat{\varphi})\vert^{2}d\xi\\
\leq & \Vert\varphi\Vert_{s}^{2}+\int_{-\infty}^{\infty}\vert(\partial_{\xi}^{3}F(t,\xi))\widehat{\varphi}+3\partial_{\xi}^{2}F(t,\xi)\partial_{\xi}\widehat{\varphi}+3\partial_{\xi}F(t,\xi)\partial_{\xi}^{2}\widehat{\varphi}+F(t,\xi)\partial_{\xi}^{3}\widehat{\varphi}\vert^{2}d\xi\\
\leq & \Vert\varphi\Vert_{s}^{2}+\int_{-\infty}^{\infty}\Bigl\vert 4it\delta\widehat{\varphi}+\Big(-\frac{6it}{(1+\vert\xi\vert)^{4}}-\frac{2t^{2}[\sgn(\xi)+2]}{(1+\vert\xi\vert)^{5}}+\frac{(-it)^{3}}{(1+\vert\xi\vert)^{6}}\Big)F(t,\xi)\widehat{\varphi}\Bigr\vert^{2}d\xi\\
 & +9\int_{-\infty}^{\infty}\Bigl\vert\frac{2it\sgn(\xi)}{(1+\vert\xi\vert)^{3}}F(t,\xi)-\frac{t^{2}}{(1+\vert\xi\vert)^{4}}F(t,\xi)\Bigr\vert\vert\partial_{\xi}\widehat{\varphi}\vert^{2}d\xi\\
 & +9\int_{-\infty}^{\infty}\Bigl\vert\frac{-it}{(1+\vert\xi\vert)^{2}}F(t,\xi)\partial_{\xi}^{2}\widehat{\varphi}\Bigr\vert^{2}d\xi+\int_{-\infty}^{\infty}\vert F(t,\xi)\partial_{\xi}^{3}\widehat{\varphi}\vert^{2}d\xi\\
\leq & \Vert\varphi\Vert_{s}^{2}+(36t^{2}+36t^{4}+t^{6})\int_{-\infty}^{\infty}\vert\widehat{\varphi}\vert^{2}d\xi+(36t^{2}+9t^{4})\int_{-\infty}^{\infty}\vert\partial_{\xi}\widehat{\varphi}\vert^{2}d\xi\\
 & +9t^{2}\int_{-\infty}^{\infty}\vert\partial_{\xi}^{2}\widehat{\varphi}\vert^{2}d\xi+\int_{-\infty}^{\infty}\vert\partial_{\xi}^{3}\widehat{\varphi}\vert^{2}d\xi\\
\leq & \Vert\varphi\Vert_{s}^{2}+(36t^{2}+36t^{4}+t^{6})\Vert\varphi\Vert_{0}^{2}+(36t^{2}+9t^{4})\Vert x\varphi\Vert_{0}^{2}+9t^{2}\Vert x^{2}\varphi\Vert_{0}^{2}+\Vert x^{3}\varphi\Vert_{0}^{2}.
\end{align*}
 
Arguing with the induction principle, we obtain the desired result.  

\end{proof}

\subsection{Some operator bounds}

\begin{lemma} \label{operator}
The operator $A= \partial_x (1+\mathcal{H}\partial_x)^{-1} \in\mathfrak{B}(\Im_{s,2}),$ 
moreover $A(u^2)\in \Im_{s,2}$ for all $\varphi\in \Im_{s,2} .$

\end{lemma}
\begin{proof}
\begin{align*}
\Vert\partial_{x}(1+\mathcal{H}\partial_{x})^{-1}\varphi\Vert_{s,2}^{2} & =\Vert\partial_{x}(1+\mathcal{H}\partial_{x})^{-1}\varphi\Vert_{s}^{2}+\Vert\partial_{x}(1+\mathcal{H}\partial_{x})^{-1}\varphi\Vert_{L_{2}^{2}}^{2}\\
 & \leq\Vert\varphi\Vert_{s}^{2}+\int_{-\infty}^{\infty}\vert x^{2}\partial_{x}(1+\mathcal{H}\partial_{x})^{-1}\varphi\vert^{2} dx\\
 & \leq\Vert\varphi\Vert_{s}^{2}+\int_{-\infty}^{\infty}\Bigl|\partial_{\xi}^{2}\Big(\frac{-i\xi}{1+\vert\xi\vert}\widehat{\varphi}(\xi)\Big)\Bigr|^{2}d\xi\\
  & \leq\Vert\varphi\Vert_{s}^{2}+\int_{-\infty}^{\infty}\Bigl| \frac{2isgn(\xi)}{(1+\vert\xi\vert)^{3}}\widehat{\varphi}(\xi)\Bigl|^{2}d\xi+\int_{-\infty}^{\infty}\Bigl|\frac{-2i}{(1+\vert\xi\vert)^{2}}\partial_{\xi}\widehat{\varphi}(\xi)\Bigl|^{2}d\xi\\
 & \ \ \ \ +\int_{-\infty}^{\infty}\Bigl|\frac{-i\xi}{1+\vert\xi\vert}\partial_{\xi}^{2}\widehat{\varphi}(\xi)\Bigl|^{2}d\xi\\
\end{align*} 
 \begin{align*}
  &\ \ \ \ \ \ \ \ \ \ \ \ \ \ \ \ \ \  \leq\Vert\varphi\Vert_{s}^{2}+4\int_{-\infty}^{\infty}\vert\widehat{\varphi}(\xi)\vert^{2}d\xi+4\int_{-\infty}^{\infty}\vert\partial_{\xi}\widehat{\varphi}(\xi)\vert^{2}d\xi+\int_{-\infty}^{\infty}\vert\partial_{\xi}^{2}\widehat{\varphi}(\xi)\vert^{2}d\xi\\
 &\ \ \ \ \ \ \ \ \ \ \ \ \ \ \ \ \ \ \leq\Vert\varphi\Vert_{s}^{2}+4\Vert\varphi\Vert_{0}^{2}+4\Vert x\varphi\Vert_{0}^{2}+\Vert x^{2}\varphi\Vert_{0}^{2}\\
 & \ \ \ \ \ \ \ \ \ \ \ \ \ \ \ \ \ \  \leq4\Vert\varphi\Vert_{s,2}^{2}.
\end{align*}

Since $s>\frac{1}{2}$ we have that, $\Im_{s,2}$ is a Banach Algebra and thefore if $u\in \Im_{s,2} $, we have that $u^2\in \Im_{s,2}$ y $\rVert Au^2\rVert_{s,2}^2<\infty$. \\
\end{proof}
In order to consider non-integer weights we have to introduce new harmonic analysis tools. Let us first recall the definition of the $A_p$ condition. We shall restrict here to $p\in(1,\infty)$ and the 1-dimensional case $\R$ (see \cite{Mu}).
 
 \begin{definition}\label{definition1} A non-negative function $w\in L^1_{loc}(\R)$ satisfies the $A_p$ inequality with $1<p<\infty\,$  if
 \begin{equation}
 \label{ap}
 \sup_{Q\;\text{interval}}\left(\frac{1}{|Q|}\int_Q w\right)\left(\frac{1}{|Q|}\int_Qw^{1-p'}\right)^{p-1}=c(w)<\infty,
 \end{equation}
 where $1/p+1/p'=1$.
 \end{definition}

  \begin{theorem}\label{theorem6} $($\cite{MuHuWh}$)$  The condition \eqref{ap} is necessary and  sufficient  for the boundedess of the Hilbert transform $\mathcal H$  in $L^p(w(x)dx)$, i.e.
 \begin{equation}
 \label{apb}
( \int_{-\infty}^{\infty}|\mathcal Hf|^pw(x)dx)^{1/p}\leq c^*\, (\int_{-\infty}^{\infty} |f|^pw(x)dx)^{1/p}.
 \end{equation}
 \end{theorem}

For further references and comments we refer to  \cite{Do} and \cite{St2}. However, even though  we will be mainly concerned with the case $p=2$, the characterization \eqref{apb} will be the one used in our proof.
In particular, one has that 
\begin{equation}
\label{cond|x|}
|x|^{\alpha}\in A_p\;\;\Leftrightarrow\;\; \alpha\in (-1,p-1).
\end{equation}

We notice that for $0\leq \theta\leq1$ we have that
\begin{equation}
\|D_x^\theta f\|_0\leq c(\|f\|_0+\|D_xf\|_0).
\label{interpolationder}
\end{equation}

and that

\begin{proposition}\label{propcomm}$($\cite{GFFLGP2}$)$
Given  $\phi\in L^{\infty}(\R)$, with $\partial_x^{\alpha}\phi\in L^2(\R)$ for $\alpha=1,2$, then
for any $\theta\in(0,1)$
\begin{equation}\label{commute}
\|J^{\theta}(\phi f)\|_2\le c_{\theta,\phi}\,\|J^{\theta}f\|_2.
\end{equation}
\end{proposition}

\section{Proof of Theorem \ref{Main1}}

\begin{proof}  We restrict to the case $r=2.$

We use Theorem \ref{pesosgrupo},  Lemma \ref{operator} and Theorem A.2 in \cite{Io3} (if $\varphi \in \Im_r$, then $x^{\alpha}\partial_{x}^{\beta}\varphi \in L^2(\mathbb{R})$, for all integers $\alpha$  and $\beta$ such that $0\leq \alpha + \beta \leq r$. Moreover, $\Vert x^{\alpha}\partial_{x}^{\beta}\varphi \Vert_0 \leq C_{\alpha , \beta}\Vert \varphi \Vert_{r} $ ).

For existence of solutions, we consider the integral equation:

\begin{equation}
\Phi(u)(t)=e^{-t\partial_{x}(1+\mathcal{H}\partial_{x})^{-1}}\varphi-\frac{1}{2}\int_{0}^{t}e^{-(t-\tau)\partial_{x}(1+\mathcal{H}\partial_{x})^{-1}}\partial_{x}(1+\mathcal{H}\partial_{x})^{-1}u^{2}(\tau)d\tau \label{exissob}
\end{equation}

Observe that $\Phi (u)\in C([0,T);\Im_{s,2})$ in (\ref{exissob}) belongs to metric space $\mathfrak{X}_{s,2}=\{u\in C([0,T);\Im_{s,2}) \ \vert \  \Vert E(t)\varphi - u(t)  \Vert_{\Im_{s,2}}\leq M\}$ with the metric $$d_{s,T}(u,v)=\sup_{t\in[0,T]}\,\rVert u(t)-v(t)\rVert_{s,2}.$$ 

In fact, (\ref{exissob}) implies that,

\begin{align*}
\Vert \phi(u)(t)-e^{-t\partial_{x}(1+\mathcal{H}\partial_{x})^{-1}}\varphi\Vert_{s,2} & \leq\int_{0}^{t}\Vert e^{-(t-\tau)\partial_{x}(1+\mathcal{H}\partial_{x})^{-1}}\partial_{x}(1+\mathcal{H}\partial_{x})^{-1}u^{2}(\tau)\Vert_{s,2}d\tau\\
 & \leq2C_{s,2}\int_{0}^{t}((t-\tau)^{2}+3)\Vert u(\tau)\Vert_{s,2}^{2}d\tau\\
 & \leq2C_{s,2}(t^{2}+3)\int_{0}^{t}\Vert u(\tau)\Vert_{s,2}^{2}d\tau\\
 & \leq2C_{s,2}(T^{2}+3)L(\Vert u\Vert_{s,2},0)\int_{0}^{t}\Vert u(\tau)\Vert_{s,2}d\tau.
\end{align*}

And if $u \in \mathfrak{X}_{s,2}(T,M) $
$$\rVert u(\tau)\rVert_{s,2}\leq \rVert u(\tau)-e^{\tau A}\varphi\rVert_{s,2}\, +\,\, \rVert e^{\tau A}\varphi \rVert_{s,2}\leq M+ \,\rVert\varphi\rVert_{s,2}.$$
Therefore, if $t\in[0,T]$
\begin{align*}
\Vert \Phi(u)(t)-e^{tA}\varphi\Vert_{s,2} & \leq2C_{s,2}(T{}^{2}+3)L(\Vert u\Vert_{s,2},0)(M+(T^{2}+3)\Vert\varphi\Vert_{s,2})T\\
 & \leq2C_{s,2}(T{}^{2}+3)^{2}L(\Vert u\Vert_{s,2},0)(M+\Vert\varphi\Vert_{s,2})T\\
 & \leq2C_{s,2}q(T)L(\Vert u\Vert_{s,2},0)(M+\Vert\varphi\Vert_{s,2})T.
\end{align*}

By Choosing $T\leq \dfrac{M}{2C_{s,2}q(T)L(\Vert u\Vert_{s,2},0)(M+\Vert\varphi\Vert_{s,2})},$\\
we have that, 
 $$\Vert \Phi(u)(t)-e^{tA}\varphi\Vert_{s,2}  \leq M$$
{\it i,e} $\Phi(u)\in \mathfrak{X}_{s,2}$.\\
 
Next,  we show that $\Phi$ is a contraction in $\mathfrak{X}_{s,2}$.

\begin{align*}
\Vert \Phi(u)(t)-\Phi(v)(t)\Vert_{s,2} & \leq\int_{0}^{t}\Vert e^{-(t-\tau)\partial_{x}(1+\mathcal{H}\partial_{x})^{-1}}\partial_{x}(1+\mathcal{H}\partial_{x})^{-1}(u^{2}-v^{2})\Vert_{s,2}d\tau\\
 & \leq2c(t{}^{2}+3)\int_{0}^{t}\Vert u^{2}(\tau)-v^{2}(\tau)\Vert_{s,2}d\tau\\
 & \leq c_{s,2}(T{}^{2}+3)\int_{0}^{t}\bigl(\Vert u(\tau)\Vert_{s,2}+\Vert v(\tau)\Vert_{s,2}\bigr)\Vert u(\tau)-v(\tau)\Vert_{s,2}d\tau\\
 & \leq c_{s,2}(T{}^{2}+3)\int_{0}^{t}\bigl(M+(t^{2}+3)\Vert\varphi\Vert_{s,2}\bigr)\Vert u(\tau)-v(\tau)\Vert_{s,2}d\tau\\
 & \leq c_{s,2}(T{}^{2}+3)\bigl(M+(t^{2}+3)\Vert\varphi\Vert_{s,2}\bigr)T\sup_{t\in[0,T]}\Vert u(t)-v(t)\Vert_{s,2}.
\end{align*}
  
If we choose $T\leq\dfrac{1}{8c_{s,2}(T{}^{2}+3)\bigl(M+(T^{2}+3)\Vert\varphi\Vert_{s,2}\bigr)}$, we obtain the desired result.\\

 Uniqueness in $\Im_{s,2}$ is a consequence of the theory in  $H^s(\mathbb{R})$.

In order to extend this local solution, let us  obtain  {\it apriori} estimates for $u$ in $\Im_{s,2}$.
  
\begin{align*}
\left\Vert u\right\Vert _{L_{2}^{2}}^{2} & =\int_{-\infty}^{\infty}x^{4}(u(x,t))^{2}dx.
\end{align*}

Then

\begin{align}
\partial_{t}\left\Vert u\right\Vert _{L_{2}^{2}}^{2} & =\partial_{t}\int_{-\infty}^{\infty}x^{4}(u(x,t))^{2}dx \label{eg}\\
 & =2\int_{-\infty}^{\infty}x^{4}u(x,t)u_{t}(x,t)dx \nonumber\\
 & =2\int_{-\infty}^{\infty}x^{4}u(x,t)\left[-\partial_{x}(1+\mathcal{H}\partial_{x})^{-1}u-\frac{1}{2}\partial_{x}(1+\mathcal{H}\partial_{x})^{-1}u^{2}\right]dx \nonumber\\
 & =-2\int_{-\infty}^{\infty}x^{4}u(x,t)\partial_{x}(1+\mathcal{H}\partial_{x})^{-1}udx-\int_{-\infty}^{\infty}x^{4}u(x,t)\partial_{x}(1+\mathcal{H}\partial_{x})^{-1}u^{2}dx \nonumber
\end{align}

Now we estimate  each integral on the right hand side of (\ref{eg}).

First we have

\begin{align}\int_{-\infty}^{\infty}x^{4}u\partial_{x}(1+\mathcal{H}\partial_{x})^{-1}udx & =\int_{-\infty}^{\infty}x^{2}ux^{2}\partial_{x}(1+\mathcal{H}\partial_{x})^{-1}udx \nonumber\\
 & =\langle x^{2}u\ ,\ x^{2}\partial_{x}(1+\mathcal{H}\partial_{x})^{-1}u\rangle_{0}\nonumber\\
 & \leq\left\Vert x^{2}u\right\Vert _{0}\left\Vert x^{2}\partial_{x}(1+\mathcal{H}\partial_{x})^{-1}u\right\Vert _{0}\nonumber\\
 &\leq 4c \Vert u\Vert _{s,2}^2\label{int1}.
\end{align}

Next, we obtain that
\begin{align}\int_{-\infty}^{\infty}x^{4}u(x,t)\partial_{x}(1+\mathcal{H}\partial_{x})^{-1}u^{2}dx & =\int_{-\infty}^{\infty}x^{2}ux^{2}\partial_{x}(1+\mathcal{H}\partial_{x})^{-1}u^{2}dx\nonumber\\
 & \leq\Vert x^{2}u\Vert_{0}\Vert x^{2}\partial_{x}(1+\mathcal{H}\partial_{x})^{-1}u^{2}\Vert_{0}\nonumber\\
 & \leq c\Vert u\Vert_{s,2}\Vert\partial_x(1+\mathcal{H}\partial_{x})^{-1}u^{2}\Vert_{s,2}\nonumber\\
 & \leq c\Vert u\Vert_{s,2}\Vert u^{2}\Vert_{s,2}\nonumber\\
 & \leq c\Vert u\Vert_{L_{2}^{2}}\alpha(T)\Vert u\Vert_{s,2}\nonumber\\
 & \leq \alpha(T)\Vert u\Vert_{s,2}^{2}\label{int2}.
\end{align}

From (\ref{int1}) and (\ref{int2}) we have

\begin{align}
\partial_{t}\left\Vert u\right\Vert _{L_{2}^{2}}^{2} & \leq 4c\Vert u\Vert_{s,2}^{2}+4c\alpha(T)\Vert u\Vert_{s,2}^{2}\nonumber\\
 & \leq\beta(T)\Vert u\Vert_{s,2}^{2}\nonumber\\
 & \leq\beta(T)\left(\Vert u\Vert_{s}^{2}+\Vert u\Vert_{L_{2}^{2}}^{2}\right)\nonumber\\
 & \leq\gamma(T)+\beta(T)\Vert u\Vert_{L_{2}^{2}}^{2},\label{estiglopes}
\end{align}

where $\gamma(T)$, \  $\beta(T)$ are constants obtained from $\Vert u\Vert_{\infty}$ and $\Vert u\Vert_{s}$ .
 Gronwall's inequelity and  (\ref{estiglopes}) imply the result.
 \end{proof}

\section{Proof of Theorem \ref{Main2}} 

\begin{proof}
  
 We will restrict ourselves to the most interesting case $2<r<\frac52$. In this setting $r=2+\theta,$ with $0<\theta<\frac{1}{2}.$ 
 
 We also consider first the local existence problem via the contraction principle as we did before. Let us fix $T>0.$
  
We estimate the group in the weighted norm:  $$\||x|^rE(t)\varphi\|_0=\||x|^\theta x^2E(t)\varphi\|_0=\|D^\theta_\xi F_2(t,\cdot,\widehat{\varphi})\|_0,$$
where we recall that
  \begin{equation*}
  \begin{aligned}
  F_2(t,\xi,\widehat{\varphi})&=\partial^2_\xi(F(t,\xi)\widehat{\varphi})\\
  &=F(t,\xi)\Big(\frac{2it\sgn(\xi)}{(1+|\xi|)^3}\widehat{\varphi} -\frac{t^2}{(1+|\xi|)^4}\widehat{\varphi}-\frac{2it}{(1+|\xi|)^2}\partial_\xi\widehat{\varphi}+\partial^2_\xi\widehat{\varphi}\Big)\\
  &=\sum_1^4F_{2,j}(t,\xi,\widehat{\varphi}).
  \end{aligned}
  \end{equation*}
  
  For $F_{2,4}$ we easily obtain from \eqref{commute}, with $\phi$ being the group $F$, that
  
   \begin{equation}\label{linear4}
  \|D^\theta_\xi F_{2,4}\|_0\leq c_T(\|\varphi\|_0+\||x|^r\varphi\|_0).
   \end{equation}

For the terms $F_{2,2}$ and $F_{2,3}$ it follows from \eqref{interpolationder}  that

\begin{equation}\label{linear2}
\begin{aligned}
\|D_\xi^\theta F_{2,2}\|_0&\leq c\, t^2(\|F_{2,2}\|_0 +\|D_\xi F_{2,2}\|_0)\\
&\leq c_T(\|\frac{F(t,\cdot)}{(1+|\cdot|)^4}\widehat{\varphi}\|_0 +\|\partial_\xi(\frac{F(t,\cdot)}{(1+|\cdot|)^4}\widehat{\varphi})\|_0)\\
&\leq c_T (\|\varphi\|_0+\|x\varphi\|_0),
\end{aligned}
\end{equation}
and

\begin{equation}
\begin{aligned}\label{linear3}
\|D_\xi^\theta F_{2,3}\|_0&\leq c \,t(\|F_{2,3}\|_0 +\|D_\xi F_{2,3}\|_0)\\
&\leq c_T(\|\frac{F(t,\cdot)}{(1+|\cdot|)^2}\partial_\xi\widehat{\varphi}\|_0 +\|\partial_\xi(\frac{F(t,\cdot)}{(1+|\cdot|)^2}\partial_\xi\widehat{\varphi})\|_0)\\
&\leq c_T (\|x\varphi\|_0+\|x^2\varphi\|_0).
\end{aligned}
\end{equation}
Last, for $F_{2,1}$ we use the continuity property of the Hilbert transform with respect to the $A_2$ weights $|x|^{2\theta},$ with $0<\theta<\frac{1}{2}$, and then proceed as we did above:

\begin{equation}\label{linear1}
\begin{aligned}
\|D_\xi^\theta F_{2,1}\|_0&= c\,t\||x|^\theta \mathcal H E(t)\psi\|_0 \\
&\leq c_T \||x|^\theta E(t)\psi\|_0\\
&\leq c_T\|D_\xi^\theta(\frac{F(t,\cdot)\widehat\varphi}{(1+|\cdot|)^3})\|_0\\
&\leq c_T(\|F(t,\cdot)\frac{\widehat\varphi}{(1+|\cdot|)^3}\|_0+\|D_\xi \frac{F(t,\cdot)\widehat\varphi}{(1+|\cdot|)^3}\|_0)\\
&\leq c_T(\|\varphi\|_0+\|x\varphi\|_0),
\end{aligned}
\end{equation}

where $\widehat\psi(\xi)=\frac{\widehat\varphi(\xi)}{(1+|\xi|)^3}.$
Therefore, collecting the information in \eqref{linear4}-\eqref{linear1}, we have that

\begin{equation}
\label{wlinearestimate}
\||x|^rE(t)\varphi\|_0\leq c_T (\|\varphi\|_0+\|\varphi\|_{L_r^2})\leq c_T\|\varphi\|_{s,r}.
\end{equation}

With the help of the former estimates and Plancherel, we now estimate the integral equation \ref{exissob} in the weighted Sobolev norm 

\begin{equation}
\begin{aligned}
\||x|^r\Phi(u)(t)\|_0&=\|D^\theta_\xi\partial^2_\xi\widehat{\Phi(u)(t)}\|_0\\
&\leq\|D_\xi^\theta F_2(t,\cdot,\widehat{\varphi})\|_0+\|D^\theta_\xi\int_0^tF_2(t-\tau,\cdot,\frac{\widehat{u\partial_xu}}{1+|\cdot|})\,\,d\tau\,\, \|_0\\
&\leq c_T\|\varphi\|_{s,r}+\int_0^T\|D_\xi^\theta F_2(t-\tau,\cdot,\frac{\widehat{u\partial_xu }}{1+|\cdot|})\|_0\,\,d\tau\\
&\leq c_T \|\varphi\|_{s,r}+\sum^4_{j=1}\int_0^T\|D_\xi^\theta F_{2,j}(t-\tau,\cdot,\frac{\widehat{u\partial_xu }}{1+|\cdot|})\|_0\,\,d\tau .
\end{aligned}
\label{contraction}
\end{equation}
We estimate each term in the  integrand similarly as we did in the linear part 
\begin{equation}
\begin{aligned}
\|D_\xi^\theta F_{2,1}(t-\tau,\cdot,\frac{\widehat{u\partial_xu }}{1+|\cdot|})\|_0&\leq
c_T(\|\frac{\xi\,\,\widehat{v}}{(1+|\xi|)^4}\|_0+\|\partial_{\xi}\big(\frac{F(t-\tau,\cdot)\sgn(\xi)i\xi\widehat{v}}{(1+|\xi|)^4}\big)\|_0)\\
&\leq c_T(\|v\|_0+ (t-\tau)\|\frac{\xi\,\,\widehat{v}}{(1+|\xi|)^6}\|_0+ \|\frac{\widehat{v}}{(1+|\xi|)^4}\|_0 \\
&+\|\frac{\xi\,\widehat{xv}}{(1+|\xi|)^4}\|_0+\|\frac{\xi\,\widehat{v}}{(1+|\xi|)^5}\|_0)\\
&\leq c_T(\|v\|_0+\|xv\|_0)\\
&\leq c_T (\|v\|_0+\|v\|_{L^2_1})\\
&\leq c_T \|u\|^2_{s,r},
\end{aligned}
\label{intt1}
\end{equation}

\begin{equation}
\begin{aligned}
\|D_\xi^\theta F_{2,2}(t-\tau,\cdot,\frac{\widehat{u\partial_xu }}{1+|\cdot|})\|_0&\leq c\,
(t-\tau)^2(\|\frac{\xi\,\widehat{v}}{(1+|\xi|)^5}\|_0+\|\partial_\xi\big(\frac{F(t-\tau,\cdot) \xi\,\widehat{v}}{(1+|\xi|)^5}\big)\|_0 )\\
& \leq (c_T(\|v\|_0 + (t-\tau)\|\frac{\xi\,\widehat{v}}{(1+|\xi|)^7}\|_0+\|\frac{\,\widehat{v}}{(1+|\xi|)^5}\|_0 \\
&+ \|\frac{\xi\,\widehat{xv}}{(1+|\xi|)^5}\|_0+\|\frac{\xi\,\widehat{v}}{(1+|\xi|)^6}\|_0)\\
&\leq c_T (\|v\|_0+\|v\|_{L^2_1})\\
&\leq c_T \|u\|^2_{s,r},
\end{aligned}
\label{intt2}
\end{equation}

\begin{equation}
\begin{aligned}
\|D_\xi^\theta F_{2,3}(t-\tau,\cdot,\frac{\widehat{u\partial_xu }}{1+|\cdot|})\|_0&\leq
(t-\tau)\|D_\xi^\theta(\frac{F(t-\tau,\cdot)}{(1+|\xi|)^2}\partial_\xi(\frac{i\xi\,\widehat{v}}{1+|\xi|}))\|_0\\
&\leq c_T (\|\frac{F(t-\tau,\cdot)}{(1+|\xi|)^2}\partial_\xi(\frac{i\xi\,\widehat{v}}{1+|\xi|})\|_0+\|\partial_\xi(\frac{F(t-\tau,\cdot)}{(1+|\xi|)^2}\partial_\xi(\frac{i\xi\,\widehat{v}}{1+|\xi|}))\|_0)\\
&\leq c_T (\|v\|_0+\|xv\|_0+\|x^2v\|_0)\\
&\leq c_T (\|v\|_0+\|v\|_{L^2_2})\\
&\leq c_T \|u\|^2_{s,r}.
\end{aligned}
\label{int3}
\end{equation}

\begin{equation}
\begin{aligned}
\|D_\xi^\theta F_{2,4}(t-\tau,\cdot,\frac{\widehat{u\partial_xu }}{1+|\cdot|})\|_0&\leq
\|J^\theta_\xi F(t-\tau,\cdot)\partial^2_\xi(\frac{i\xi\,\widehat{v}}{1+|\xi|})\|_0\\
&\leq c_T \| J^\theta_\xi \partial^2_\xi(\frac{i\xi\,\widehat{v}}{1+|\xi|})\|_0\\
&\leq c_T (\|J^\theta_\xi\frac{\xi}{1+|\xi|}\partial_\xi^2\widehat{v}\|_0 +\|J^\theta_\xi\frac{1}{(1+|\xi|)^2}\partial_\xi\widehat{v}\|_0+\|J^\theta_\xi\frac{\sgn(\xi)\,\widehat{v}}{(1+|\xi|)^3}\|_0)\\
&\leq c_T (\|J^\theta_\xi\partial_\xi^2\widehat{v}\|_0 +\|\frac{1}{(1+|\xi|)^2}\partial_\xi\widehat{v}\|_0+\|\partial_\xi(\frac{1}{(1+|\xi|)^2}\partial_\xi\widehat{v})\|_0\\
&+\|\langle x\rangle^\theta\mathcal H\psi\|_0)\\
&\leq c_T(\|v\|_{L^2_r} +\|v\|_{L^2_1}+\|v\|_{L^2_2}+\|\langle x\rangle^\theta\psi\|_0)\\
&\leq c_T (\|v\|_0+\|v\|_{L^2_r}+ \|J^\theta_\xi\frac{\widehat{v}}{(1+|\xi|)^3}\|_0)\\
&\leq c_T (\|v\|_0+\|v\|_{L^2_r}+\|\frac{\widehat{v}}{(1+|\xi|)^3}\|_0+\|\partial_\xi(\frac{\widehat{v}}{(1+|\xi|)^3})\|_0)\\
&\leq c_T \|u\|^2_{s,r},
\end{aligned}
\label{int4}
\end{equation}

 where $\widehat{\psi}=\frac{\widehat{v}}{(1+|\xi|)^3}$ and $v=u^2$.
 
The rest of the proof follows exactly the same arguments as in the proof of Theorem \ref{integercase}, so we omit the details.\end{proof}
 
\section{Proof of Theorem \ref{Main3}} 
\vskip .2in

By hypothesis we consider $u\in C([0,T]:\Im_{s,2})$ the solution of the rBO initial value problem with  initial data, $\varphi,$ having mean value $\int \varphi(x)\,dx\geq 0.$ We also assume that for  $t_1=0<t_2<T$ it holds  $u(t_j)\in \Im_{s,\frac52}, j=1,2.$

 Again, we write $r=\frac52=2+\frac12$. Representing the solution $u$ with the help of Duhamel's formula in Fourier space we have that
 \begin{equation}
\begin{aligned}
\partial_\xi^2\widehat u(t,\xi)= \sum^4_{j=1}F_{2,j}(t,\xi,\widehat{\varphi})+\int_0^t\, F_{2,j}(t-\tau,\cdot,\frac{\widehat{u\partial_xu }}{1+|\cdot|})\,\,d\tau).
\end{aligned}
\label{ill2der}
\end{equation}

It can easily be checked from the hypothesis on the initial data and  by employing the same arguments as in \eqref{linear4}-\eqref{linear3} and \eqref{intt1}-\eqref{int3} that 

\begin{equation}
D_\xi^\frac12(\sum_{j=2}^4F_{2,j}(t,\cdot,\varphi)+\sum_{j=1}^3\int_0^t\, F_{2,j}(t-\tau,\cdot,\frac{\widehat{u\partial_xu }}{1+|\cdot|})\,\,d\tau )\in L^2(\R),
\label{termsok}
\end{equation}

for every $t\in[0,T].$

Therefore it follows from \eqref{ill2der} and \eqref{termsok} that
 
\begin{equation}
\label{conclusion1}
D_\xi^\frac12(F_{2,1}-\int_0^t F_{2,4}\,d\tau)=D_\xi^{\frac12}\Big( F(t,\xi)\frac{2it\sgn(\xi)}{(1+|\xi|)^3}\widehat{\varphi}-\frac12\int_0^t\,F(t-\tau,\xi)\partial^2_\xi(\frac{i\xi\widehat{v}(\tau,\xi)}{1+|\xi|})\,\,d\tau \Big)
\end{equation}

 belongs to $L^2(\R)$ at time $t=t_2.$
 
Let us introduce a cutoff function $\chi\in C^{\infty}_0(\R)$, supp $\chi\subseteq [-2\epsilon,2\epsilon]$ and $\chi\equiv 1$
in $(-\epsilon,\epsilon)$ and split $F_{2,1}$ as 
\begin{equation}
\label{}
\begin{aligned}
F_{2,1}(t,\xi,\varphi)&=F(t,\xi)\frac{2it\sgn(\xi)\widehat{\varphi}(\xi)}{(1+|\xi|)^3}(\chi(\xi)+(1-\chi(\xi)))\\
&=G_1+G_2.
\end{aligned}
\end{equation}

With the help of the usual estimates it is easy to verify that $D_\xi^\frac12 G_2\in L^2(\R)$ at any $t\in[0,T]$.

Next, we consider $G_1$

\begin{equation}
\begin{aligned}
G_1&=2it\chi(\xi)\sgn(\xi)\widehat{\varphi}(\xi)+2it\chi(\xi)\sgn(\xi)\widehat{\varphi}(\xi)(\frac{F(t,\xi)}{(1+|\xi|)^3}-1)\\
&=G_{1,1}+G_{1,2}.
\end{aligned}
\end{equation}

Since $\frac{F(t,\xi)}{(1+|\xi|)^3}-1$ vanishes at $\xi=0$ and it is bounded by 2, it follows that $D_\xi^\frac12 G_{1,2}\in L^2(\R)$ for any $t\in[0,T]$ as well.

The same strategy can be applied by subtracting  $\widehat\varphi (0)$ and therefore we can write up $G_1$ as

\begin{equation}
\label{conclusion2}
G_1=2it\chi(\xi)\sgn(\xi)\widehat{\varphi}(0)+G,
\end{equation}

where $D_\xi^\frac12 G\in L^2(\R)$ for any $t\in[0,T].$

We proceed in a similar fashion with the integral part in \eqref{conclusion1} and conclude that 
\begin{equation}
\label{conclusion3}
\frac12\int_0^t\,F(t-\tau,\xi)\partial^2_\xi(\frac{i\xi\widehat{v}(\tau,\xi)}{1+|\xi|})\,\,d\tau=i\chi(\xi)\sgn(\xi)\,\int_0^t\,\widehat{v}(\tau,0)\,d\tau+H,
\end{equation}

where $D_\xi^\frac12 H\in L^2(\R)$ for any $t\in[0,T].$

To end the proof we notice that from \eqref{conclusion1}, \eqref{conclusion2} and \eqref{conclusion3} it follows that

\begin{equation}
\label{conclusion3}
D_\xi^\frac12 (\sgn(\xi)\chi(\xi)(2t_2\widehat{\varphi}(0)+\int_0^{t_2}\,\widehat{v}(\tau,0)\,d\tau))\in L^2(\R),
\end{equation}

and therefore as a consequence of a pointwise homogeneous derivative introduced in \cite{St} we obtain from Proposition 3 in \cite{FoPo}  that
\begin{equation}
\label{conclusion4}
2t_2\widehat{\varphi}(0)+\int_0^{t_2}\,\widehat{v}(\tau,0)\,d\tau=0.
\end{equation}

We use now the fact that $v=u^2$ and rewrite \eqref{conclusion4} as
\begin{equation}
\label{conclusion5}
2t_2\widehat{\varphi}(0)+\int_0^{t_2}\,\int_{-\infty}^\infty\,u^2(\tau,x)\,dx\,d\tau=0.
\end{equation}

The continuity of the solution $u$ allows us to conclude from \eqref{conclusion5} that $u\equiv 0,$ and this completes the proof.

\end{document}